\documentclass[12pt]{amsart}
\usepackage{graphicx}
\usepackage[centertags]{amsmath}
\usepackage{amsfonts}
\usepackage{amssymb}
\usepackage{amsthm}
\usepackage{comment}
\usepackage{fullpage}
\usepackage{url}
\newtheorem{thm}{Theorem}
\newtheorem{lem}{Lemma}[section]
\newtheorem{cor}{Corollary}[section]

\newtheorem{prop}[lem]{Proposition}
\theoremstyle{definition}

\theoremstyle{remark}
\newtheorem{rem}{Remark}[section]
\numberwithin{equation}{section}

\newcommand{\norm}[1]{\left\Vert#1\right\Vert}

\newcommand{\FF}[1]{\mathbb{F}_{\!#1}}

\newcommand{\G}{\Gamma}

\newcommand{\calF}{\mathcal{F}}

\newcommand{\calH}{\mathcal{H}}

\newcommand{\calO}{\mathcal{O}}

\newcommand{\bbZ}{\mathbb{Z}}
\newcommand{\Z}{\mathbb{Z}}
\newcommand{\bbA}{\mathbb{A}}
\newcommand{\bbQ}{\mathbb{Q}}
\newcommand{\bbR}{\mathbb R}
\newcommand{\bbC}{\mathbb C}
\newcommand{\bbG}{\mathbb G}

\newcommand{\bbF}{\mathbb F}

\newcommand{\Tr}{ \mathrm{tr}}

\newcommand{\SL}{ \mathrm{SL}}
\newcommand{\SO}{ \mathrm{SO}}

\newcommand{\SU}{ \mathrm{SU}}

\newcommand{\GL}{ \mbox{GL}}

\newcommand{\Sp}{\mathrm{Sp}}

\newcommand{\Mat}{\mathrm{Mat}}

\newcommand{\spec}{\mathrm{Spec}}

\newcommand{\vol}{\mathrm{vol}}

\newcommand{\lap}{\triangle}
\newcommand{\bs}{\backslash}
\newcommand{\id}{1\!\!1}

\newcommand{\frakp}{\mathfrak{p}}
\newcommand{\frakq}{\mathfrak{q}}
\newcommand{\fraka}{\mathfrak{a}}

\newcommand{\frakd}{\mathfrak{d}}

\newcommand{\frakg}{\mathfrak{g}}

\newcommand{\OF}{\calO_F}

\newcommand{\kp}{{\rm{k}_\frakp}}
\newcommand{\kb}{{\bar{\rm{k}}_\frakp}}
\DeclareMathOperator{\Lie}{Lie}
\DeclareMathOperator{\Res}{Res}
\newcommand{\Dq}{\Delta(\frakq)}
\newcommand{\Gq}{\Gamma(\frakq)}
\newcommand{\Gqp}{\Gamma(\frakq')}
\newcommand{\Vq}{V(\frakq)}
\newcommand{\Rq}{R_\frakq}
\newcommand{\GG}{\mathbf{G}}

\begin{document}
\title[]{A uniform spectral gap for congruence covers of a hyperbolic manifold}%
\author{Dubi Kelmer and Lior Silberman}%
\address{Department of Mathematics, University of Chicago,  5734 S. University
Avenue Chicago, Illinois 60637}
\email{kelmerdu@math.uchicago.edu}
\address{University of British Columbia, Department of Mathematics
1984 Mathematics Road Vancouver   BC   V6T 1Z2}
\email{lior@math.ubc.ca}
%\thanks{The first author was partially supported by the NSF grant DMS-1001640.}%


\date{\today}%
\dedicatory{}%
\commby{}
\begin{abstract}
Let $G$ be $\SO(n,1)$ or $\SU(n,1)$ and let $\Gamma\subset G$
denote an arithmetic lattice.
The hyperbolic manifold $\Gamma\bs \calH$ comes with a natural family of covers, coming from the congruence subgroups of $\Gamma$.
In many applications, it is useful to have a bound for the spectral gap that is uniform for this family.
When $\Gamma$ is itself a congruence lattice, there are very good bounds coming from known results
towards the Ramanujan conjectures. In this paper, we establish an effective bound that is uniform for
congruence subgroups of a non-congruence lattice.
\end{abstract}
\maketitle
\section*{Introduction}
Let $G$ be a connected almost simple Lie group, let $K\subset G$ be a maximal
compact subgroup and let $\Gamma \subset G$ be a lattice.  We write
$X=\Gamma\bs G/K=\Gamma\bs \calH$ for the corresponding locally symmetric
space, endowed with the hyperbolic Riemannian metric coming from the
Killing form on $\Lie(G)$.  Let $\lap$ denote the (non-negative)
Laplace-Beltrami operator on $\calH$ and $X$, and
let $\spec(\Gamma)$ denote the point spectrum
of this operator on $L^2(X)$ (that is, the set of eigenvalues with
eigenfunctions in $L^2(X)$).  Since $\Gamma$ is a lattice, the constant
function is an eigenfunction and $0=\lambda_0\in \spec(\Gamma)$.  This
eigenvalue is simple by the maximum principle, and we accordingly write
$\lambda_1(\Gamma)$ for the smallest positive member of
$\spec(\Gamma)$.

We say that the lattice $\Gamma$ has \emph{property $(\tau)$} with respect
to a family of finite index subgroups, if $\lambda_1(\Gamma')$ is uniformly
bounded away from zero for all subgroups $\Gamma'\subset \Gamma$ in this
family.
When $\Gamma$ is an \emph{arithmetic lattice}, it has a natural family of
finite-index subgroups, its \emph{congruence subgroups} (defined below).
We are interested in establishing property $(\tau)$ with respect to this family.

\begin{rem}
It was essentially shown by Brooks \cite{Brooks86} that the definition
of property $(\tau)$ given above is equivalent to the condition
that the trivial representation is isolated (in the Fell topology)
from the other irreducible representations of $G$ occurring in
$L^2(\Gamma'\bs G)$ for all $\Gamma'$ in the family. The latter
condition makes sense also for semisimple groups of higher rank,
and this is the usual definition of property $(\tau)$.
See the books \cite{LubotzkyZuk,Lubotzky10} for more details.
\end{rem}

When $G$ has property $(T)$ (the trivial representation is isolated in the
full unitary dual), there is a uniform spectral gap for \emph{all}
locally symmetric spaces covered by $G/K$, in particular for the family
of all finite-index subgroups of $\Gamma$.  This is the case for groups $G$
with real rank greater than one, and also for the rank one groups
$\Sp(n,1)$ and $F_4^{-20}$.

In the remaining cases of $\SO(n,1)$ and $\SU(n,1)$, not only
is there no uniform gap varying over all lattices, but the constructions
of Randol \cite{Randol74} and Selberg \cite{Selberg65}
further show that arithmetic lattices never have property $(\tau)$
with respect to the family of all finite index subgroups.

The starting point of our work is the fact that when $\Gamma$ is a
\emph{congruence lattice} (defined below), it has property $(\tau)$
with respect to its family of congruence subgroups.
This was first established by Selberg for
$\SL_2(\Z) \subset \SL_2(\bbR)\cong \SO(2,1)$, and later in many cases by
Burger and Sarnak \cite{BurgerSarnak91}.  Finally Clozel \cite{Clozel03}
obtained this in full generality.
Moreover, for congruence lattices, there are uniform
(and in many cases very good) explicit bounds for the spectral gap
coming from known bounds towards the Generalized Ramanujan conjectures
in the general linear group.

Much less is known when $\Gamma$ is arithmetic (but not congruence).
For the special cases of $\SO(2,1)$ and $\SO(3,1)$, Sarnak and Xue
\cite{SarnakXue91} gave an elementary argument showing that any
arithmetic lattice satisfies property $(\tau)$ with respect to its
congruence subgroups.  Their argument produces explicit bounds that are
uniform for the part of the spectrum orthogonal to the spectrum of the
original group.  In this paper, we will generalize their method to obtain
a similar result for all arithmetic lattices in $\SO(n,1)$ and $\SU(n,1)$.
In particular, our result establishes property $(\tau)$ for the family
of (most) congruence subgroups of a fixed arithmetic lattices.
We note, however, that
our result is no longer elementary: in order to establish property $(\tau)$
for arithmetic lattices, we rely on known property $(\tau)$ for
congruence lattices.

\begin{rem}
The results on property $(\tau)$ for congruence lattices also apply for an (irreducible) congruence lattice in a semi-simple group of the form  $G=\prod_j G_j$. Moreover, in this case the result is even stronger, and it implies that for any irreducible representation $\pi=\bigotimes_j\pi_j$ occurring in $L^2(\Gamma\bs G)$ each factor $\pi_j$ is uniformly bounded away from the trivial representation. This stronger notion is referred to as a \emph{strong spectral gap}; see \cite{kelmerSarnak09} for details. In \cite{kelmer10gap}, the first author extended the method of Sarnak and Xue to give a uniform bound for the \emph{strong spectral gap} for congruence subgroups of a (potentially non-congruence) irreducible lattice in a product $\SO(2,1)^d$ for $d\geq 2$. It might be possible to use the results of this paper in order to obtain a similar result for a product of arbitrary rank one groups.
\end{rem}

We now set the notation required for stating our main result.
We parameterize the eigenvalue $\lambda$ in the standard form
$\lambda_s = \rho^2 - s^2$ with $s\in i\bbR\cup (0,\rho]$ and
$\rho=\tfrac{n-1}{2}$ for $\SO(n,1)$, $\rho=\tfrac{n}{2}$
for $\SU(n,1)$. With this notation, we consider $\spec(\Gamma)$
as a subset of $i\bbR\cup (0,\rho]$.

Let $F\subset \bbR$ denote a totally real number field with a fixed
infinite place.  Let $\bbG\subset \GL_n$ be a closed $F$-subgroup, such that
$G = \bbG(\bbR)$ is isomorphic to $\SO(n,1)$ or $\SU(n,1)$, while in all other
infinite places $\bbG(F_v)$ is compact.

Let $\OF \subset F$ denote the ring of integers of $F$ and set
$\Delta = \bbG(F) \cap \GL_n(\OF)$ (by abuse of notation we commonly
write $\bbG(\OF)$ for this group).
Now, for any ideal $\frakq \subset \OF$, we have the
\emph{principal congruence group}
\[\Dq=\{\gamma \in \Delta |\; \gamma \equiv I\pmod{\frakq}\}.\]

By definition, a lattice $\Gamma$ is a \emph{congruence group} if
it contains $\Dq$ for some $\frakq$, and it is an \emph{arithmetic} lattice if
it is commensurable with $\Delta$.  The definition of $\Delta$ above
depends on the specific choice of embedding of $\bbG$ in a general
linear group, but the commensurability class and the notion of a
congruence subgroup are independent of this choice.

With these notation, the result of Clozel on property $(\tau)$ can be
stated as follows: There is a positive constant $\alpha=\alpha(\bbG)$
such that for all ideals $\frakq$
\begin{equation}\label{e:sgcong}
\spec(\Delta(\frakq))\subset i\bbR\cup (0,\rho-\alpha)\cup\{\rho\}.
\end{equation}
\begin{rem}
The Generalized Ramanujan Conjectures would imply that $\alpha=\min\{1,\rho\}$.
For the orthogonal groups, the method of Burger and Sarnak
\cite{BurgerSarnak91}, together with the results of
\cite{KimSarnak03,BlomerBrumley10_preprint} on the automorphic spectrum
of $\GL_2(\bbA)$, imply that $\alpha\geq \frac{25}{64}$ for $\SO(2,1)$,
and that $\alpha\geq\frac{25}{32}$ for $\SO(n,1),\;n\geq 3$.
These bounds also hold for certain arithmetic lattices in unitary groups.
However, in general the best known bounds for the unitary groups are due to
Clozel \cite{Clozel03}, and come from lifting representations to
$\GL_n(\bbA)$ and using the bounds towards the Ramanujan Conjecture due to Luo,
Rudnick and Sarnak \cite{LuoRudnickSarnak99}.
\end{rem}

Fixing an arithmetic lattice $\Gamma$ we call the following its
\emph{congruence subgroups}:
\[\Gq=\Gamma \cap \Dq.\]
In what follows, we will assume without loss of generality
(by replacing $\Gamma$ with $\Gamma(1)$) that $\Gamma$ is a finite index
subgroup of $\Delta$.
When $\frakq'|\frakq$ we have $\Gq \subset \Gqp$
so we can identify $L^2(\Gqp\bs \calH)$ with the subspace
of left-$\Gqp$-invariant functions in $L^2(\Gq\bs\calH)$.
We write $L^2(\Gq\bs\calH)_\textrm{old}$ for the sum of these
subspaces as $\frakq'$ ranges over the (finite set) of ideals properly
dividing $\frakq$, and $L^2(\Gq\bs\calH)_\textrm{new}$
for its orthogonal complement.
Both spaces are preserved by the Laplace-Beltrami operator and we write
$\spec(\Gq)_\textrm{new}$ for the set of eigenvalues in
$L^2(\Gq\bs\calH)_\textrm{new}$.

We can now state our main result.
\begin{thm}\label{t:main}
Let $\Gamma\subset \bbG(\OF)$ be of finite index.
There is an ideal $\frakd=\frakd(\Gamma)$, such that for any $\epsilon>0$
there is a constant $q_0=q_0(\Gamma,\epsilon)$ such that
$$ \spec(\Gq)_{\rm{new}} \subseteq i\bbR\cup [0,\rho-\eta\alpha+\epsilon],$$
for all ideals $\frakq$ prime to $\frakd$ of norm at least $q_0$.
Here, $\alpha=\alpha(\bbG)$ is as in \eqref{e:sgcong} and $\eta=\eta(G)$
is an explicit constant given by
\begin{equation}\label{e:eta}\eta(G)=\left\lbrace\begin{array}{cc}
\frac{4}{3n(n+1)} & G\cong \SO(n,1),\;2\leq n<6\\
\frac{4(n-2)}{3n(n+1)} & G\cong \SO(n,1),\;n\geq 6\\
\frac{2}{3(n+2)} & G\cong \SU(n,1),\; n\geq 1
\end{array}\right.\end{equation}
\end{thm}
As an immediate corollary we get that $\Gamma$ satisfies property $(\tau)$
with respect to its congruence subgroups of levels prime to $\frakd$.
More precisely,
\begin{cor}
There is a positive constant $\delta>0$ (depending on $\Gamma$) such that
$\lambda_1(\Gamma(\frakq))>\delta$ for all ideals $\frakq$ prime to $\frakd$.
\end{cor}
%\begin{rem}
%We note that the results of Bourgain, Gamburd, Varju,  Sarnak
%\cite{BourgainGamburdSarnak10a,BourgainGamburdSarnak10}
%and Bourgain and Varj{\'u} \cite{BourgainVarju10_preprint,Varju10_preprint}, together with the results of
%Pyber and Szab\'o \cite{PyberSzabo10_preprint}, and Breuillard, Green and Tao
%\cite{BreuillardGreenTao10_preprint}, also establish property $(\tau)$
%for congruence subgroups. In fact, these results apply also when $\Gamma\subset \bbG(\OF)$ is of
%infinite index (as long as it is geometrically finite and Zariski dense).  However, this method is
%not effective, and hence does not provide an explicit bound for
%the spectral gap.
%\end{rem}

\begin{rem}
Following the results of Breuillard, Green, and Tao
\cite{BreuillardGreenTao10_preprint} and Pyber and Szab\'o \cite{PyberSzabo10_preprint}, the work of Salehi Golsefidy and Varj\'u \cite{SalehiVarju_preprint, Varju10_preprint} (generalizing the results of Bourgain, Gamburd, and Sarnak  \cite{BourgainGamburdSarnak10a,BourgainGamburdSarnak10}) also establishes property $(\tau)$ with respect to congruence subgroups of square free levels.
However, this method does not produce an explicit bound for the spectral gap.
In fact, \cite{SalehiVarju_preprint,Varju10_preprint} applies also when $\Gamma\subset \bbG(\OF)$ is of
infinite index, as long as it is geometrically finite and Zariski dense.
See also \cite{BourgainVarju10_preprint} for similar results for $\Gamma\subset\SL_n(\bbZ)$ and congruence subgroups of arbitrary level.
\end{rem}

Before we explain the strategy of our proof, let us describe the
original result of Sarnak and Xue in more detail.
For any eigenvalue $\lambda$ we denote by $m(\lambda,\Gamma)$ the multiplicity
of $\lambda$ in $\spec(\Gamma)$.
Using the Selberg trace formula, de George and Wallach
\cite{DeGeorgeWallach78} showed that for $\Gamma$ co-compact
$m(\lambda,\Gamma)\lesssim\vol(\Gamma\bs G)$ and that
$m(\lambda,\Gamma)=o(\vol(\Gamma\bs G))$ when $\lambda<\rho$.
In \cite{Xue91}, Xue sharpened this result for congruence subgroups,
showing that for $\lambda<\rho$,
\[m(\lambda,\Gq)\lesssim \Vq^{1-\mu(\lambda)},\]
where $\Vq=[\Gamma:\Gq]$ and
$0<\mu(\lambda)<\frac{1}{n^3}$ is a small constant.

In \cite{SarnakXue91}, Sarnak and Xue conjectured that the correct
bound should be
\begin{equation}\label{e:SX}
m(\lambda_s,\Gq)\lesssim_\epsilon \Vq^{\frac{\rho-s}{\rho}+\epsilon}.
\end{equation}
They related this conjecture to a Diophantine conjecture, concerning the lattice point counting function
\[N(\Gq,T)=\{\gamma\in \Gq|\;\;d(\gamma o,o)\leq T\},\]
where $o\in\calH$ is the fixed point of $K$ and $d(x,y)$ is the hyperbolic distance.
In particular, they conjectured that
\begin{equation}\label{e:SX2}N(\Gq,T)\lesssim_\epsilon \frac{e^{2\rho T(1+\epsilon)}}{\Vq}+e^{\rho T},\end{equation}
uniformly in $T$ and $\frakq$, and showed that this conjecture implies the multiplicity bound (\ref{e:SX}).

For $\SO(2,1)$ and $\SO(3,1)$, they managed to establish the bound
(\ref{e:SX2}) by an elementary counting argument, thus proving \eqref{e:SX}
in these cases.  Finally, for $\SO(2,1)$ (respectively $\SO(3,1)$) they
showed that for eigenvalues in $\spec(\Gq)_{\rm{new}}$, the spectral
multiplicity $m(\lambda,\Gq)$ is bounded below by $\Vq^{\frac{1}{3}-\epsilon}$
(respectively  $\Vq^{\frac{2}{3}-\epsilon}$).
Combining this bound with \eqref{e:SX}, it followed that if
$s<\rho-\tfrac{1}{6}$, then $\lambda_s\not\in \spec(\Gamma(\frakq))_{\rm{new}}$
when $\Vq$ was sufficiently large.

In order to prove Theorem \ref{t:main}, we will follow the same
general strategy.  However, proving the Diophantine bound \eqref{e:SX2}
in general by a direct counting argument seems out of reach.
Instead we will take a different approach, using spectral theory and decay
of matrix coefficients (similar arguments are used in
\cite{BourgainKontorovichSarnak10, GorodnikNevo09}).
Our method does not give the conjectured bound \eqref{e:SX2}.  Instead, using \eqref{e:sgcong} for the spectrum of congruence lattices,
we get the following bound:
\begin{thm}\label{t:counting}
$N(\Gq,T)\lesssim \frac{e^{2\rho T}}{\Vq}+e^{(2\rho -\alpha)T}$
uniformly in $T$ and $\frakq$, where $\alpha$ is as in \eqref{e:sgcong}.
\end{thm}
Using this weaker bound for the counting function, one gets the following
intermediate bound for the multiplicities.
\begin{thm}\label{t:multiplicity}
$m(\lambda_s,\Gq)\lesssim \Vq^{\frac{\rho-s}{\alpha}}$ for $s\in (0,\rho)$.
\end{thm}
\begin{rem}
For $G=\SO(2,1)$ or $\SO(3,1)$, the Selberg Conjecture reads $\alpha=\rho$.
Assuming this, we recover the bounds of \eqref{e:SX} and \eqref{e:SX2}.
In general, however, $\alpha$ is much smaller then $\rho$ even conjecturally,
and Theorem \ref{t:multiplicity} is trivial when $s<\rho-\alpha$ is small.
That said, when $s$ is sufficiently close to $\rho$ our bound is comparable
to the conjectured bound in \eqref{e:SX} and suffices for our purposes.
\end{rem}
The final ingredient we need is the lower bound for the multiplicities of
new eigenvalues. For this we show
\begin{thm}\label{t:multiplicity2}
There is an ideal $\frakd$, depending on $\Gamma$, such that for all ideals
$\frakq$ prime to $\frakd$, all $\lambda\in \spec(\Gamma(\frakq))_{\rm{new}}$
satisfy
$$m(\lambda,\Gq)\gtrsim_\epsilon V(\frakq)^{\eta-\epsilon},$$
where $\eta$ is as in \eqref{e:eta}.
\end{thm}
\begin{rem}
If we consider only square-free ideals we can remove the condition that
$\frakq$ be prime to $\frakd$ and at the same time slightly improve the bound,
replacing $\eta$ by $\tfrac{3}{2}\eta$.
\end{rem}
Theorem \ref{t:main} is now a direct consequence of
Theorems \ref{t:multiplicity} and \ref{t:multiplicity2}.
\begin{rem}
Using the conjectured bound (\ref{e:SX}) for the multiplicities,
instead of the bound established in Theorem \ref{t:multiplicity}, would allow
us to replace $\alpha\eta$ by $\rho\eta$ in Theorem \ref{t:main}.
We note that when $n$ is large, this gives very strong bounds on the spectrum.
In particular, for the orthogonal groups when $n\geq 16$, assuming
\eqref{e:SX} implies a spectral gap that is stronger then the best known
bounds even in the congruence case.
\end{rem}

\subsection*{Acknowledgements} 
We thank P.~Varj\'u for explaining his results.
The first author was partially supported by the NSF grant DMS-1001640.
The second author was partially supported by an NSERC Discovery Grant.

\section{Background and notation}
We write $A\lesssim B$ or $A=O(B)$ to indicate that $A\leq cB$
for some constant $c$.  If we wish to emphasize that constant depends on
some parameters we use subscripts, for example
$A\lesssim_\epsilon B$.  Note that all the implied constants in this paper
may depend on the group $\bbG$, which we consider fixed.  Sometimes they
only depend on the group $G$; in that case they are expressed via the
number $n$.  We also write $A\asymp B$ to indicate that
$A\lesssim B\lesssim A$. Finally, the cardinality of a finite set $S$
will be denoted $\# S$ or $|S|$.

\subsection{Lie theory}
$G$ will be an almost simple Lie group of real rank 1; $K\subset G$ will be
a maximal compact subgroup so that $\calH=G/K$ is a rank one symmetric space.
The Killing form on the Lie algebra $\frakg$ of $G$
induces a $G$-invariant Riemannian structure on $\calH$ with
%the corresponding distance function $d(\cdot,\cdot)$ and
the Riemannian measure $dx$. In particular, we will consider the cases of
$G=\SO(n,1)$ and $G=\SU(n,1)$ corresponding to real and complex hyperbolic
spaces respectively.

Let $G=NAK$ be an Iwasawa decomposition of $G$, with $A$ a maximal
diagonalizable subgroup and $N$ a maximal unipotent subgroup.
Let $\fraka$ denote the Lie algebra of $A$ and let $\fraka_\bbC^*$
denote its complexified dual (one dimensional since $G$ is of rank one).
Fix a positive Weyl chamber, denote by $\rho$ the half sum of the positive
roots, and fix a positive $X\in \fraka$ of norm one with respect to the
Killing form.  We now identify $\fraka_\bbC^*$ with $\bbC$ via their values
at $X$ and, slightly abusing notation, write $\rho=\rho(X)\in \bbR$ so that
$\rho=\frac{n-1}{2}$ in the case of $\SO(n,1)$ and $\rho=\tfrac n 2$ in the case of
$\SU(n,1)$.

We normalize the Haar measure $dk$ of $K$ to have total mass $1$ and the
Haar measure $dg$ on $G$ so that it projects to $dx$ on $\calH = G/K$.
We have the Cartan decomposition $G=KA^+K$ with $A^+$ the closure of
the positive Weyl chamber. Let $X\in \fraka$ be as above,
then any $g\in G$ can be expressed in the form
$g=k_1(g)\exp(t(g)X)k_2(g)$ with $k_j\in K$ and $t(g)=d(g o,o)\in \bbR^+$.
The Haar measure with respect to this decomposition is given by
$dg=D(t) dk_1 dt dk_2$ with $D(t)\asymp e^{2\rho t}$
(see \cite[Proposition 5.28]{Knapp86}).

\begin{comment}
Assume we have an embedding $G\subset \GL_n(\bbR)$ and let $\Theta$ be a
Cartan involution of $\GL_n(\bbR)$ which restricts to the Cartan involution
of $G$ which fixes $K$.  We write $g^\mathrm{t}$ for $\Theta(g^{-1})$,
so that $k^\textrm{t} = k^{-1}$ for $k\in K$ and $a^\textrm{t} = a$
for $a\in A$.  With this data $\norm{g}^2=\Tr(g^\textrm{t} g)$ is known as
the norm of $g\in G$.  It is easy to check that
$\norm{g} = \norm{\exp(t(g)X)}$, from which it follows that
$\norm{g} \asymp e^{t(g)\varpi(X)}$ where $\varpi$ is the highest weight of
the given representation.  For the standard $n+1$-dimensional representation
of $SO(n,1)$ (resp.\ the standard (2n+2)-dimensional representation
of $SU(n,1)$) the highest weight is $1$, so we have $\norm(g) \asymp e^{t(g)}$
(for a closed-form evaluation of $\norm{g}$ in the case $G=\SO(n,1)$ see
\cite[page 214]{SarnakXue91}).
It will be more natural here to use the map given by the Adjoint
representation if $G$.  It is not an embedding, strictly speaking, but
this is not an issue.  The highest weight is still $1$.
Using the $KA^+K$ decomposition and integrating we can compute the volume
of balls in this norm. In particular, their asymptotic growth rate is given by
$$\lim_{T\to\infty}
\frac{\log\int_{\norm{g}\leq T} dg}{\log(T)} = 2\rho.$$
\end{comment}

\subsection{Spectral theory}
Denote by $\hat{G}$ the unitary dual of $G$ and by
$\hat{G}^1$ the spherical dual.
We parameterize $\hat{G}^1$ by $\fraka_\bbC^*/W$ where $W$ is the Weyl group.
With this parametrization the tempered representations lie in $i\fraka^*_\bbR$
and the non-tempered representations are contained in $(0,\rho]$.
We use the notation
\[\hat{G}^1=\{\pi_s|s\in i\bbR^+\cup [0,\rho]\},\]
where the representations $\pi_s, s\in i\bbR^+$ are the (tempered)
principal series representations, the representations
$\pi_s, s\in (0,\rho)$ are the (non-tempered) complementary series,
and $\pi_\rho$ is the trivial representation.

We recall the relation between Laplace eigenvalues and irreducible
representations.  For this let $\Omega$ be the Casimir operator of $G$, $\Omega$ acts on any irreducible representation $V_\pi$ by
scalar multiplication $\Omega v+\lambda(\pi)v=0$.
We may normalize $\Omega$ so that the restriction of $\Omega$ to the space
of right-$K$-invariant functions on $G$ coincides with
the Laplace-Beltrami operator $\lap$ on $\calH=G/K$.
With this normalization we have $\lambda(\pi_s)=\rho^2-s^2$.

\subsection{Decay of matrix coefficients}
For each $s\in i\bbR^+\cup [0,\rho]$ consider the \emph{spherical function}
defined by
\begin{equation}\label{e:spherical}
\phi_s(g)=\langle \pi_s(g)v,v\rangle
\end{equation}
where $v\in V_{\pi_s}$ is the normalized spherical vector.
The asymptotic behavior of this function is well-known.
For $s\in(0,\rho]$ the function $\phi_s$ is positive and decays like
\[\phi_s(\exp(tX))\asymp C_s e^{(s-\rho)t}.\]
For $s\in i\bbR$, $\phi_s$ is oscillatory and decays like
\[|\phi_{s}(\exp(tX))|\lesssim  t e^{-\rho t}.\]
As a consequence, we have the following bounds on general matrix
coefficients:
Let $\pi$ denote a unitary representation of $G$ on a Hilbert space $V_\pi$
such that all non-trivial non-tempered $\pi_s$ weakly contained in $V_\pi$
satisfy $s\in [0,\rho-\alpha]$. Then for any smooth spherical vector
$v\in V_\pi$,
\begin{equation}\label{e:decay}|\langle \pi(g)v,v\rangle| \leq C(\alpha) \norm{v}^2 e^{-\alpha t(g)}.\end{equation}

\section{Proof of Theorem \ref{t:counting}}
Since $\Gq\subset\Dq$, we have $N(\Gq;T)\leq N(\Dq;T)$,
and it is sufficient to bound the latter, that is to show that
\begin{equation}\label{e:counting}
\sharp \{\gamma\in \Dq|t(\gamma)\leq T\} \lesssim
\frac{e^{2\rho T}}{\Vq}+e^{(2\rho-\alpha)T}.
\end{equation}
It is easier to obtain this bound in a smooth form.  Accordingly for
$g,h\in G$ let
\[N(\Dq,g,h;T) =
\sharp \{\gamma\in \Dq|t(h^{-1}\gamma g)\leq T\}.\]
Fix a fundamental domain $\calF_1$ for $\Delta\bs G$ (containing the identity)
and let $B\subset \calF_1$ be a small fixed ball.
Let $\psi$ denote a smooth positive bi-$K$ invariant function on $G$ of
total mass one that is supported on $B$.
Consider the averaged counting function
\begin{equation}\label{e:defpsi}
\Psi(T)=\int_B\int_B N(\Dq,g,h;T)\psi(g)\psi(h) dgdh.
\end{equation}
Since $\psi$ is supported inside a fundamental domain $\calF_1$ it
is also supported in a fundamental domain $\calF_\frakq$ for
$\Dq\bs G$.  Let
$\tilde\psi(g)=\sum_{\gamma\in \Gq}\psi(\gamma g)$
denote the periodized function on $\Dq\bs G$.
\begin{lem}
\begin{equation}\label{e:Psi}
\Psi(T)=\int_G \id_T(g)\langle \Rq(g)\tilde\psi,\tilde\psi\rangle dg,
\end{equation}
where $\id_T(g)$ is the indicator function of the set $\{g\in G|t(g)\leq T\}$
and $\Rq$ is the right regular representation of $G$ on
$L^2(\Dq\bs G)$.
\end{lem}
\begin{proof}
Since $\psi$ is supported on $B\subset \calF_\frakq$ we can replace the
integral over $B\times B$ in \eqref{e:defpsi} with an integral over
$\calF_\frakq\times\calF_\frakq$. Writing the counting function as
\[N(\Dq,g,h;T) =
\sum_{\gamma\in \Gq} \id_T(h^{-1}\gamma g),\]
we get
\[\Psi(T)=\int_{\calF_q}\int_{\calF_q}\left(\sum_{\gamma\in \Gq}\id_T(h^{-1}\gamma g)\right)\tilde\psi(g)\tilde\psi(h) dgdh.\]
Changing variables $g\mapsto \gamma^{-1}g$ and using the
$\Gq$-invariance of $\tilde{\psi}$ we get
\begin{eqnarray*}
\Psi(T)%&=&\int_{\calF_\frakq}\int_{\calF_\frakq}\left(\sum_{\gamma\in \Gq}\id_T(h^{-1}\gamma g)\right)\tilde\psi(g)\tilde\psi(h) dgdh\\
%&=& \int_{\calF_\frakq}\left(\sum_{\gamma\in \Gq}\int_{\calF_\frakq}\id_T(h^{-1}\gamma g)\tilde\psi(g)dg\right)\tilde\psi(h) dh\\
&=& \int_{\calF_\frakq}\left(\sum_{\gamma\in \Gq}\int_{\gamma F_\frakq }\id_T(h^{-1}g)\tilde\psi(g)dg\right)\tilde\psi(h) dh\\
&=& \int_{\calF_\frakq}\left(\int_G \id_T(h^{-1}g)\tilde\psi(g)dg\right)\tilde\psi(h)dh.\\
%&=&\int_{\calF_\frakq}\left(\int_G \id_T(g)\tilde\psi(hg)dg\right)\tilde\psi(h)dh\\
%&=& \int_G \id_T(g)\left(\int_{F_\frakq}\tilde\psi(hg)\tilde\psi(h)dh\right)dg\\
%&=&\int_G \id_T(g)\langle \Rq(g)\tilde\psi,\tilde\psi\rangle_{\Gq\bs G}dg\\
\end{eqnarray*}
Making one more change of variables $g\mapsto hg$ and changing the order
of integration we get
\begin{eqnarray*}
\Psi(T)&=& \int_{\calF_\frakq}\left(\int_G \id_T(g)\tilde\psi(hg)dg\right)
     \tilde\psi(h)dh\\
&=& \int_G \id_T(g)\left(\int_{\calF_\frakq}\tilde\psi(hg)\tilde\psi(h)dh\right)dg\\
&=& \int_G \id_T(g)\langle \Rq(g)\tilde\psi,\tilde\psi\rangle dg.\\
\end{eqnarray*}
\end{proof}

We can now use the results about decay of matrix coefficients to
bound $\Psi(T)$:
\begin{lem}
$\Psi(T)\lesssim \frac{e^{2\rho T}}{\Vq}+ e^{(\rho-\alpha)T}.$
\end{lem}
\begin{proof}
Decompose $\tilde\psi=\psi_0+\psi^\bot$,
with
$\psi_0=\frac{1}{\vol(\Dq\bs G)}1$ the projection of $\tilde\psi$ onto the constant function and $\psi_0^\bot\in L^2_0(\Dq\bs G)$ its orthogonal complement. We then have
\[\langle \Rq(g)\tilde\psi,\tilde\psi\rangle=||\psi_0||^2+\langle \Rq(g)\psi_0^\bot,\psi_0^\bot\rangle.\]
Since $\Dq$ is a congruence group, all non-tempered $\pi_s\in \hat{G}^1$ (weakly) contained in $L^2_0(\Dq\bs G)$ satisfy $s\in [0,\rho-\alpha)$. Consequently, we get that
\[\langle \Rq(g)\psi^\bot,\psi^\bot \rangle\lesssim ||\psi^\bot||^2 e^{-\alpha t(g)}.\]
We have $||\psi_0||^2=\frac{1}{\vol(\Dq\bs G)}\asymp \frac{1}{\Vq}$ and we can bound $||\psi^\bot||^2\leq ||\tilde\psi||^2=\int_B|\psi(g)|^2dg$ which is independent on $q$. Plugging this in \eqref{e:Psi} gives
\[\Psi(T)\lesssim \int_0^T(\frac{1}{\Vq}+e^{-\alpha t})e^{2\rho t}dt\lesssim \frac{e^{2\rho T}}{\Vq}+ e^{(2\rho-\alpha)T}.\]
\end{proof}

We can now conclude the proof of (\ref{e:counting}), and hence of Theorem \ref{t:counting}. Since on average
\[\Psi(T)=\int_B\int_B N(T,\Dq;g,h)\psi(g)\psi(h) dgdh\lesssim \frac{e^{2\rho T}}{\Vq}+ e^{(2\rho-\alpha)T},\]
there is a point $(g,h)\in B\times B$ for which
\[N(T,\Dq;g,h)\lesssim \frac{e^{2\rho T}}{\Vq}+ e^{(2\rho-\alpha)T}.\]
Now, since $B$ is compact, there is some $\delta=\delta(B)$ (not depending on $\frakq$) such that for any $h,g\in B$
\[t(\gamma)-\delta\leq t(h^{-1}\gamma g)\leq t(\gamma)+\delta.\]
Replacing $T$ by $T\pm \delta$ (which only affects the implied constant) we get (\ref{e:counting}).

\section{Proof of Theorem \ref{t:multiplicity}}
Theorem \ref{t:multiplicity} follows from Theorem \ref{t:counting}
by the same arguments as in the proof of \cite[Theorem 3]{SarnakXue91}
(and \cite{HuntleyKatznelson93} for the non-compact case).
For the sake of completeness we include the details below.

Fix an eigenvalue $\lambda=\rho^2-s^2$ and let $V_\lambda(\Gq)$ denote
the corresponding eigenspace.  Fix an orthonormal basis $\psi_1,\ldots,\psi_m$
for $V_\lambda(\Gq)$ and consider the Bergman function
\[B(x;\frakq,\lambda)=\sum_{j=1}^m |\psi_j(x)|^2.\]
Note that
$$B(x;\frakq,\lambda)=\sup_{\psi\in V_\lambda(\Gq)} \frac{|\psi(x)|^2}{\norm{\psi}^2},$$
is independent on the choice of basis. Since the left action of
$\Gamma$ preserves $V_\lambda(\Gq)$, the function $B(x;\frakq,\lambda)$
is $\Gamma$ invariant and
\[\int_{\calF_1}B(x;\frakq,\lambda)dx=\frac{m(\lambda,\frakq)}{\Vq}.\]
In the case of non-compact $\calF_1$, \cite[Lemma 2.1]{HuntleyKatznelson93}
obtains a compact neighborhood of the identity $\calF_0\subset \calF_1$
(independent of $\frakq$) such that
\begin{equation}\label{e:collar}
\int_{\calF_0}B(x;\frakq,\lambda)dx\gtrsim \frac{m(\lambda,\frakq)}{\Vq}.
\end{equation}
uniformly for $\lambda\leq \rho^2-1$,
where the implied constant depends only on $\calF_0$.

Let $f(g)=\id_T(g)\phi_s(g)$ with $\id_T$ as above and $\phi_s$ the spherical function defined in (\ref{e:spherical}). Let $F=f*\check{f}$ where $\check{f}(g)=\bar{f}(g^{-1})$.  By \cite[Lemma 2.1]{SarnakXue91} the function $F(g)$ is bi-$K$ invariant and satisfies
\begin{equation} F(g)\lesssim\left\lbrace\begin{array}{cc} e^{2sT}e^{-\rho t(g)} & \mbox{ if } t(g)\leq 2T\\
0 & \mbox{ if } t(g)> 2T\end{array}\right.\end{equation}
Its spherical transform is given by $\hat{F}(s)=|\hat{f}(s)|^2$ where
\begin{eqnarray*}
\hat{f}(s)=\int_G \id_T(g)|\phi_s(g)|^2dg\asymp \int_0^{T} e^{2\rho t} e^{2(s-\rho)t} dt\asymp e^{2sT},
\end{eqnarray*}
so that $\hat{F}(s)\asymp e^{4sT}$.

Identifying $\Gq\bs \calH=\Gq\bs G/K$ and using the pre-trace formula we have
\[\sum_{\gamma\in \Gq}F(x^{-1}\gamma x)=\sum_{k} \hat{F}(s_k)|\psi_k(x)|^2 +\mathcal{E},\]
where the sum on the right is over a basis of Laplacian eigenfunctions in $L^2(\Gq\bs \calH)=L^2(\Gq\bs G/K)$ with eigenvalue $\lambda_{k}=\rho^2-s_k^2$ and $\mathcal{E}$ corresponds to the contribution of the continuous spectrum (this is a finite sum of integrals of $\hat{F}$ against Eisenstein series).
From positivity of $\hat{F}$ we can bound
\[\hat{F}(s)B(x;\frakq,\lambda_s)\leq \sum_{\gamma\in \Gq}F(x^{-1}\gamma x).\]
From compactness of $\calF_0$ there is $\delta$ depending only on $\calF_0$ (not on $\frakq$) such that
\[t(\gamma)-\delta\leq t(x^{-1}\gamma x)\leq t(\gamma)+\delta,\]
for all $x\in \calF_0$. We thus get that
\begin{eqnarray*}
\hat{F}(s)B(x;\frakq,\lambda_s)%&\leq& \sum_{\gamma\in \Gq}F(x^{-1}\gamma x)\\
&\lesssim& e^{2sT}\mspace{-18.0mu}\sum_{\substack{\gamma\in\Gq \\ t(x^{-1}\gamma x)\leq 2T}} e^{-\rho t(x^{-1}\gamma x)}\\
&\lesssim& e^{2sT}\sum_{\substack{\gamma\in\Gq \\ t(\gamma)\leq 2T}}e^{-\rho t(\gamma)},
\end{eqnarray*}
uniformly on $\calF_0$. Integrating by parts, and using Theorem \ref{t:counting} we get
\begin{eqnarray*}
\hat{F}(s)B(x;\frakq,\lambda_s) &\lesssim & e^{2sT}\int^{2T} e^{-\rho t}N(\Gq;t)dt\\
&\lesssim&  e^{2sT}\int^{2T} e^{-\rho t}(\frac{e^{2\rho t}}{\Vq}+e^{(2\rho-\alpha)t})dt\\
&\lesssim& e^{2sT}(\frac{e^{2\rho T}}{\Vq}+e^{2(\rho-\alpha)T}).
\end{eqnarray*}
Dividing by $\hat{F}(s)\asymp e^{4sT}$ and integrating over $\calF_0$ we get
\begin{eqnarray*}\frac{m(\lambda_s,\Gq)}{\Vq}\lesssim \int_{\calF_0}B(x;\frakq,\lambda_s)dx
\lesssim e^{-2sT}(\frac{e^{2\rho T}}{\Vq}+e^{2(\rho-\alpha)T}).
\end{eqnarray*}
Now multiply by $\Vq$ and take $T=\frac{\log(\Vq)}{2\alpha}$ to get $m(\lambda_s,\Gq)\lesssim \Vq^{\frac{(\rho-s)}{\alpha}}$.

\section{Proof of Theorem \ref{t:multiplicity2}.}\label{s:eG}
The lower bound for the multiplicities of eigenvalues in
$\spec(\Gq)_{\rm{new}}$ follows from a lower bound on the dimensions
of irreducible representations of the finite groups $\Gq\bs\Gamma$.
We say that a representation of $\Gq\bs\G$ is \emph{new}
if it does not factor through one of the quotients $\Gqp\bs\G$, for any proper
divisor $\frakq'$ of $\frakq$; equivalently, we say that a representation of
$\Gamma$ is of level $\frakq$ if it factors through a new representation
of $\Gq\bs\G$.

Since $\Gq$ is normal in $\Gamma$, the left action of $\Gamma$ preserves
$L^2(\Gamma(\frakq)\bs \calH)$ and commutes with the Laplacian.
We thus get a representation of $\Gamma(\frakq)\bs \Gamma$ on each
eigenspace $V_\lambda(\Gq)$.
Moreover, $\lambda\in \spec(\Gq)$ belongs to $\spec(\Gq)_{\textrm{new}}$
iff there is at least one invariant subspace in $V_\lambda(\Gq)$ on which
$\Gamma/\Gq$ acts via a new representation.
Consequently, we can reduce Theorem \ref{t:multiplicity2} to the following
result regarding the irreducible representations of $\Gq\bs\Gamma$.

\begin{prop}\label{p:irrep1}
There is an ideal $\frakd$ depending on $\Gamma$ such that all
irreducible representations of level $\frakq$ prime to $\frakd$ satisfy
\[\dim(\rho)\gtrsim_\epsilon \Vq^{\eta-\epsilon}.\]
\end{prop}

\subsection{Reduction to prime powers}
The Strong Approximation Theorem of Weisfeiler \cite[Thm.\ 1]{Weisfeiler84}
(together with the Borel Density Theorem) show that there exists an
ideal $\frakd_1$ in $\OF$ and a group scheme $\GG$ defined over the
localization of $\OF$ away from $\frakd_1$ such that $\GG_F = \bbG$,
and such that for all $\frakq$ prime to $\frakd_1$ there exists an isomorphism
$$\Gq\bs \Gamma \cong \GG(\OF/\frakq) \cong \prod_{j=1}^{\omega(\frakq)} \GG(\OF/\frakp_j^{r_j}),$$
where $\frakq = \prod_{j} \frakp_j^{r_j}$ is the prime factorization of $\frakq$.  It immediately follows that
$$V(\frakq)=|\Gq\bs \Gamma|=\prod_{j=1}^{\omega(\frakq)}|\GG(\OF/\frakp_j^{r_j})|.$$

Further, any irreducible representation $\rho$ of $\GG(\OF/\frakq)$
is isomorphic to a tensor product of irreducible representations
$\rho_j$ of $\GG(\OF/\frakp_j^{r_j})$.  We thus also have
$$\dim(\rho)=\prod_{j=1}^{\omega(\frakq)} \dim(\rho_j). $$
Finally, $\rho$ is of level $\frakq$ if and only if the $\rho_j$ are of
levels $\frakp_j^{r_j}$.
Since $|\GG(\OF/\frakp^r)|\asymp |\kp|^{r\dim(G)}$ and the number
$\omega(\frakq)$ of prime divisors of $\frakq$ satisfies
$2^{\omega(\frakq)}\lesssim_\epsilon |\OF/\frakq|^\epsilon$,
we can reduce Proposition \ref{p:irrep1} to the following result
regarding the irreducible new representations of $\GG(\OF/\frakp^r)$.
\begin{prop}\label{p:irrep2}
There is an ideal $\frakd$ (contained in $\frakd_1$) such that for all prime ideals $\frakp$
not dividing $\frakd$, any new representation $\rho$ of $\GG(\OF/\frakp^r)$
satisfies
\begin{equation*}
\dim\rho\gtrsim |\kp|^{\frac{2re}{3}},
\end{equation*}
where $\kp=\OF/\frakp$ is the corresponding residue field and
\begin{equation}\label{e:eG}
e=e(\GG)=\frac{3\eta\dim(\GG)}{2}=\left\lbrace\begin{array}{cc}
1 & \GG(\bbR)\cong \SO(n,1),\;n<6\\
n-2 & \GG(\bbR)\cong \SO(n,1),\;n\geq 6\\
  n & \GG(\bbR)\cong \SU(n,1)\\
  \end{array}\right.
\end{equation}
\end{prop}

\subsection{The prime case}
Fix a prime ideal $\frakp\subset\calO_F$ not dividing $\frakd_1$ and
let $\kp=\calO_F/\frakp$ denote the residue field.
The group scheme $\GG_\kp$ is then an algebraic group defined over $\kp$
and $\GG(\calO_F/\frakp)=\GG(\kp)$ is its set of $\kp$-rational points,
which is a quasisimple group of Lie type.  The smallest dimension of an
irreducible representation of such groups was studied in
\cite{Landazuri72,LandazuriSeitz74,SeitzZalesskii93}.  We quote here
the bounds in \cite{SeitzZalesskii93} that are relevant to our case.

\begin{itemize}
\item When $G \cong \SO(n,1)$ and $n$ is even,
$\GG(\kp)\cong\SO_{n+1}(\kp)$ is the split orthogonal group of that rank
over $\kp$.  If $n$ is odd, then $\GG(\kp)\cong \SO_{n+1}^{\pm}(\kp)$
is one of the two orthogonal groups.  In these cases, the dimension of any nontrivial irreducible representation of $\GG(\kp)$
is bounded below by $c|\kp|^{n-2}$ for $n\geq 6$ and by $c|\kp|$ for $n<6$.

\item When $G$ is a unitary group, either $\GG(\kp)\cong \SU_{n+1}(\kp)$ or
$\GG(\kp)\cong \SL_{n+1}$ depending on whether $\frakp$ is inert or split
in a corresponding quadratic extension of $F$ (we include the finite set
of ramified primes in the idea $\frakd$).  In these cases, the dimension
of any nontrivial irreducible representation of $\GG(\kp)$ is bounded below
by $c|\kp|^{n}$.
\end{itemize}

Consequently, we get that in all cases the representations of prime level
$\frakp$ satisfy $\dim(\rho)\gtrsim |\kp|^e$, establishing Proposition
\ref{p:irrep2} for prime ideals.

\subsection{The prime power case}
We fix a prime ideal $\frakp$ not dividing $\frakd_1$ and an integer $r>1$;
write $k=[\frac{r}{2}]$ for the integer part of $r/2$.
We will also assume from here on that the characteristic, $p$, of $\kp$ does not ramify in $F$ and that it is not
\emph{bad} for $G$  (see \cite[I \S 4]{SpringerSteinberg70} for details). For this we only need to exclude a finite number
of primes, which we include in the ideal $\frakd$.
Under these assumptions, we will show that all irreducible representations
of level $\frakp^r$ satisfy
\begin{equation}\label{e:irrep2}
\dim(\rho)\gtrsim |\kp|^{2ke}=\left\lbrace\begin{array}{cc} |\kp|^{re} & r=2k\\ |\kp|^{(r-1)e} & r=2k+1\end{array}\right.
\end{equation}

Consider the finite local ring $\calO=\OF/\frakp^r$.
Slightly abusing notation, we will denote by $\frakp=\frakp\calO$
the maximal ideal of $\calO$.  Since $\frakp$ is prime to $\frakd_1$,
$\calO$ is a quotient of $(\OF)_{\frakd_1}$ and we may replace $\GG$
with $\GG_\calO$ (which we denote $\GG$ from now on).
For any $\tfrac{r}{2} \leq l\leq r$ denote by $\GG(\frakp^l)$ the
kernel of the projection from $\GG(\calO)$ to $\GG(\calO/\frakp^l)$.
Then $\GG(\frakp^l)$ is commutative.  Moreover, if $\pi\in \calO$ is
a uniformizer for $\frakp$, then the map $I+\pi^lX \mapsto X$ is an
isomorphism of $\GG(\frakp^l)$ with $\frakg(\calO/\frakp^{k-l})$ where
$\frakg=\Lie(\GG)$ (cf.\ \cite[Lemma 5.2]{Weisfeiler84}).

For any $X\in \frakg(\calO)$ denote by $C_{G(\calO/\frakp^l)}(X)$
and $C_{\frakg(\calO/\frakp^l)}(X)$ the centralizers of $X$ in
$G(\calO/\frakp^l)$ and $\frakg(\calO/\frakp^l)$ respectively.
In particular, for $l=1$ we have,
\[C_{G(\kp)}(X)=\{g\in G(k)| gX=Xg\},\quad C_{\frakg(\kp)}(X)=\{Y\in \frakg(\kp)| YX=XY\}.\]
\begin{prop}\label{p:irrep3}
For any representation $\rho$ of $\GG(\OF/\frakp^r)=\GG(\calO)$
of level $\frakp^r$ there is a nontrivial $X\in \frakg(\kp)$ such that
$$\dim(\rho)\geq \frac{|\GG(\kp)||\frakg(\kp)|^{k-1}}{|C_{\GG(\kp)}(X)||C_{\frakg(\kp)}(X)|^{k-1}}.$$
\end{prop}
\begin{proof}
The group $\GG(\frakp^{r-k})$ is a commutative normal subgroup of
$\GG(\calO)$ and the adjoint action of
$\GG(\calO/\frakp^k)=\GG(\calO)/\GG(\frakp^{k})$ on $\GG(\frakp^{r-k})$
induces a corresponding co-adjoint action on its dual $\widehat{\GG(\frakp^{r-k})}$.

Let $\rho$ denote a new representation of $\GG(\calO)$.
The restriction of $\rho$ to $\GG(\frakp^{r-k})$ is thus supported on a
$\GG(\calO/\frakp^k)$-invariant set in $\widehat{\GG(\frakp^{r-k})}$.
In this parametrization, that $\rho$ is new is equivalent to the
restriction of $\rho$ to $\GG(\frakp^{r-1})$ being non-trivial.  It follows
that there is at least one character $\chi\in \widehat{\GG(\frakp^{r-k})}$
in the support that is not trivial on $\GG(\frakp^{r-1})$, and
$$\dim\rho\geq \#\{\chi^g|g\in \GG(\calO/\frakp^k)\}=\frac{|\GG(\calO/\frakp^k)|}{\#\{g\in \GG(\calO/\frakp^k)|\chi^g=\chi\}}.$$

%Next, since $\GG(\frakp^{r-k})\cong \frakg(\calO/\frakp^k)$ we can identify $\widehat{\GG(\frakp^{r-k})}$ with
%$\widehat{\frakg(\calO/\frakp^k)}$.
We construct an isomorphism of $\frakg(\calO/\frakp^k)$ with $\widehat{\frakg(\calO/\frakp^k)}=\widehat{\GG(\frakp^{r-k})}$ that is compatible with the co-adjoint action.
Let $B(X,Y)=\Tr(\rm{ad}(X)\rm{ad}(Y))$ denote the Killing form on
$\frakg(\calO/\frakp^k)$. On $\frakg(\calO/\frakp)=\frakg(\kp)$, the Killing form is non-degenerate and invariant under
the adjoint action. By induction, this is also true on $\frakg(\calO/\frakp^k)$.
Indeed, if there is $X\in \frakg(\calO/\frakp^k)$ such that $B(X,Y)\equiv 0 \pmod{\frakp^k}$ for all $Y\in \frakg(\calO/\frakp^k)$, then in particular it is true modulo $\frakp$. Since over $\kp$ the Killing form is non-degenerate, this implies that $X\equiv 0\pmod{\frakp}$ and $X=uX'$ with $u\in \frakp$. Dividing by $u$ we get that $B(X',Y)\equiv 0 \pmod{\frakp^{k-1}}$ for all $Y\in \frakg(\calO/\frakp^{k-1})$, and by induction $X\equiv 0\pmod{\frakp^k}$.

Since we assume that the characteristic $p$ of $\kp$ is unramified in $F$, the group $\calO/\frakp^k \cong(\bbZ/p^k\bbZ)^f$ is a free $\bbZ/p^k\bbZ$
module of rank $f$ (the inertia degree).
Fixing a basis, the action of multiplication by $a\in\calO/\frakp^k$
gives an embedding $\calO/\frakp^k\hookrightarrow\Mat(f,\bbZ/p^k\bbZ)$
and we denote by $\Tr(a)\in \bbZ/p^k\bbZ$ the trace of the
corresponding matrix. The map $\Tr:\calO/\frakp^k\to\bbZ/p^k\bbZ$
is surjective and satisfies that $\Tr(ab)\equiv 0\pmod{p^k}$ for all $b\in \calO/\frakp^k$ if and only if
$a\equiv 0\pmod{\frakp^k}$. We now construct the isomorphism $\frakg(\calO/\frakp^k)\cong \widehat{\frakg(\calO/\frakp^k)}$ sending $X\in \frakg(\calO/\frakp^k)$ to the character $\chi_X\in \widehat{\frakg(\calO/\frakp^k)}$ given by
$\chi_X(Y)=\exp(\frac{2\pi i\Tr(B(X,Y))}{p^k})$.

With the above identification, the co-adjoint action of $\GG(\calO/\frakp^k)$ on $\widehat{\frakg(\calO/\frakp^k)}$ is identified with the adjoint action of $\GG$ on its Lie algebra, indeed $\chi_X^g(Y)=\chi_X(gYg^{-1})=\chi_{g^{-1}Xg}(Y)$.
Moreover, notice that $\chi_{X}$ is trivial on $\GG(\frakp^{r-1})\cong \frakg(\frakp^{k-1}/\frakp^k)$
if and only if $X\equiv 0\pmod {\frakp}$.  Consequently, we get that there is $X\in \frakg(\calO/\frakp^k)$ with $X\not\equiv 0\pmod{\frakp}$
such that
$$\dim\rho\geq \frac{|\GG(\calO/\frakp^k)|}{|C_{\GG(\calO/\frakp^k)}(X)|}=\frac{|\GG(\kp)||\frakg(\kp)|^{k-1}}{|C_{\GG(\calO/\frakp^k)}(X)|}.$$

Finally, we can write
\[|C_{\GG(\calO/\frakp^k)}(X)|=\sum_{a\in C_{\GG(\calO/\frakp^{k-1})}(X)}\#\{g\in C_{\GG(\calO/\frakp^k)}(X)|g\equiv a\pmod{\frakp^{k-1}}\}.\]
If $g_1,g_2\in C_{\GG(\calO/\frakp^k)}(X)$ and
$g_1\equiv g_2\pmod{\frakp^{k-1}}$ then $g_1g_2^{-1}=I+\pi^{k-1}Y$ with
$Y\in C_{\frakg(\calO/\frakp)}(X)$.  Hence, the number of such elements
for each $a$ is bounded by $|C_{\frakg(\kp)}(X)|$ and by induction we find
that
$|C_{\GG(\calO/\frakp^k)}(X)|\leq |C_{\GG(\kp)}(X)||C_{\frakg(\kp)}(X)|^{k-1}$
concluding the proof.
\end{proof}

Since $|\GG(\kp)|\gtrsim |\kp|^{\dim(G)}$ and $|\frakg(\kp)|=|\kp|^{\dim(G)}$, the bound \eqref{e:irrep2} follows from the following estimate on the
cardinality of centralizers.
\begin{prop}\label{p:centralizer}
Let $X\in \frakg(\kp)$ denote a nontrivial element. Then,
$$|C_{\frakg(\kp)}(X)|\leq |\kp|^{\dim(G)-2e}\mbox{ and }\quad|C_{\GG(\kp)}(X)|\lesssim |\kp|^{\dim(\GG)-2e}.$$
\end{prop}
\begin{proof}
Let $G=\GG_\kp$, so that $G$ is one of the groups $\SL_{n+1},\SU_{n+1}$ or
$\SO_{n+1}^\pm$ over the finite field $\kp$, and let $\frakg=\Lie(G)$ be
its Lie algebra.  Let $\kb$ be an algebraic closure of $\kp$.
Then $G=G(\kb)$ is an algebraic group defined over $\kp$ and for any
$X\in \frakg(\kp)$, the centralizer $C_{G}(X)$ is an algebraic subgroup
of $G$ whose $\kp$ points are precisely $C_{G(\kp)}(X)$.
Moreover, by \cite[Corollary 5.2]{SpringerSteinberg70} we have that
$C_\frakg(X)$ is the Lie algebra of $C_G(X)$.

It is known that the number of $\kp$-rational points of any connected
algebraic group $H$ defined over $\kp$ is roughly $|\kp|^{\dim H}$ --
a precise bound is
$$(|\kp|-1)^{\dim(H)}\leq |H(\kp)|\leq (|\kp|+1)^{\dim{H}},$$
(cf. \cite[Lemma 3.5]{Nori87}).
The group $C_G(X)$ is not necessarily connected.  However, by
\cite[Chapter II, 4.1 and 4.2]{SpringerSteinberg70}) and
\cite[Chapter IV, 2.26]{SpringerSteinberg70} its number of connected
components is uniformly bounded.  Consequently, to bound the size of the
centralizer it is sufficient to bound the dimension, that is, to show
\[\dim(C_\frakg(X))\leq \dim(G)-2e.\]
We will show that this holds for any $X\in\frakg(\kb)$; the statement
is now purely about the algebraic group $G$ defined over an
algebraically closed field.

Recall the Jordan decomposition in the Lie algebra; for any $X\in \frakg$
there is a unique decomposition $X=X_s+X_n$ with $X_s$ semi-simple and
$X_n$ nilpotent which commute with each other.  Since this decomposition is
unique, $C_G(X)=C_G(X_s)\cap C_G(X_n)$.  Consequently, we may assume that
$X$ is either semisimple or nilpotent.

We start with $G=\SL_{n+1}$ (corresponding to $\GG(\bbR)$ being a unitary
group) so that $\dim G=n^2+2n$ and $e(G)=n$.
Let $X\in \frakg$ denote a non-trivial semi-simple element.
We may assume that $X$ lies in the Lie algebra of the standard split torus,
in which case it is easy to determine the centralizer explicitly. A simple computation shows that if
$r_1,\ldots, r_d$ are the multiplicities of the $d$ distinct eigenvalues
of $X$, so that $r_1+\ldots +r_d=n+1$, then
\[\dim C_\frakg(X)=\sum_{j=1}^d r_d^2-1.\]
This is clearly maximal when there are only two eigenvalues of
multiplicities $r_1=1$ and $r_2=n$ respectively, so that
\[\dim C_\frakg(X)\leq n^2\leq\dim{G}-2e.\]

Next, let $X\in \frakg$ be a nontrivial nilpotent element
and let $r_j,\;j=1,\ldots,d$ denote the number of blocks of size $j$
in its Jordan normal form.  Then $\sum_j jr_j=n+1$ and
\[\dim C_\frakg(X)=\sum_{j=1}^d(r_j+\ldots+ r_d)^2-1,\]
(see \cite[Chapter IV, 1.8]{SpringerSteinberg70}).
This is maximal when $r_1=n-1$ and $r_2=1$ so that
\[\dim C_\frakg(X)\leq n^2-1\leq\dim(G)-2e.\]

We consider the orthogonal group, $G=\SO_{n+1}$ with its canonical inclusion
in $\SL_{n+1}$.  We have $\dim(G)=\frac{n(n+1)}{2}$ as well as $e=(n-2)$
for $n\geq 6$ and $e=1$ for $n<6$.
Let $X\in\frakg$ be a semisimple element.  Since $X^t=-X$,
the nonzero eigenvalues of $X$ come in pairs $\lambda,-\lambda\in \kb^\times$
each having the same multiplicity.  Let $r_0=\dim\ker(X)$ and let
$r_1,\ldots,r_d$ denote the multiplicities of distinct pairs of nonzero
eigenvalues, so that $r_0+2(r_1+\ldots+r_d)=n+1$.  For such an element,
we can compute the centralizer explicitly, and its dimension is given by
\[\dim C_\frakg(X)=\frac{r_0(r_0-1)}{2}+\sum_{j=1}^d r_j^2.\]
Clearly, this is maximal if there is only one pair of nonzero eigenvalues
(with multiplicity $r_1$) and the zero eigenvalue with multiplicity
$r_0=n+1-2r_1$, in which case the dimension is
$$\dim C_\frakg(X)=\frac{(n+1)n}{2}+r_1(3r_1-2n-1).$$
For $n\neq 3,5$ this is maximal when $r_1=1$ and
$$\dim C_\frakg(X)\leq \frac{(n+1)n}{2}-2(n-1)\leq \dim(G)-2e.$$
When $n=3,5$ the maximal dimension is obtained when $r_1=\frac{n+1}{2}$
and it is $4$ and $9$ respectively, in particular it is bounded by $\dim(G)-2$.

Finally, for a non-trivial nilpotent element $X\in\frakg(\kb)$,
again write it in its Jordan normal form
(see \cite[IV, \S 2.19]{SpringerSteinberg70} for the normal form in the
orthogonal group). Denote by $r_j$ the number of blocks of size $j$,
then $\sum jr_j=n+1$ and $r_j$ is even for even $j$. With this data we have
\[\dim C_\frakg(X)=\frac{1}{2}\sum_j(r_j+\ldots+r_d)^2-\frac{1}{2}\sum_{j \mbox{ odd}}r_j.\]
It is not hard to see that this is maximal when $r_1=n-2$ and $r_3=1$ so that
\[\dim C_\frakg(X)\leq \frac{n(n+1)}{2}-2(n-1)\leq\dim(G)-2e. \]
\end{proof}

To conclude, Propositions \ref{p:irrep3} and \ref{p:centralizer}
give the bound \eqref{e:irrep2}.
We thus get that a representation $\rho$ of level $\frakp^r$ satisfies
$\dim(\rho) \gtrsim |\kp|^{er}$ for even $r$, and
$\dim(\rho) \gtrsim |\kp|^{e(r-1)}\geq |\kp|^{2re/3}$ for odd $r\geq 3$.
Together with the prime case discussed above, this concludes the proof of Proposition \ref{p:irrep2}, and hence of Theorem \ref{t:multiplicity2}.

----------------------------------------------------------------
%%GATHER{C:/Users/Dubi/Documents/tex/bib/Mybib.bib}   % For Gather Purpose Only
%\bibliographystyle{amsalpha}
%\bibliography{Mybib}
%\bibliography{C:/Users/Dubi/Documents/tex/bib/Mybib}
%%\bibliography{/zf/kelmerdu/Documents/tex/Bib/Mybib}

\def\cprime{$'$} \def\cprime{$'$}
\providecommand{\bysame}{\leavevmode\hbox to3em{\hrulefill}\thinspace}
\providecommand{\MR}{\relax\ifhmode\unskip\space\fi MR }
\providecommand{\MRhref}[2]{%
  \href{http://www.ams.org/mathscinet-getitem?mr=#1}{#2}
}
\providecommand{\href}[2]{#2}

\end{document}